\documentclass[reqno, 10pt, centertags,draft]{amsart}
\usepackage{amsmath,amsthm,amscd,amssymb,latexsym,upref,stmaryrd}

\newtheorem{theorem}{Theorem}
\newtheorem{lemma}[theorem]{Lemma}
\newtheorem{remark}[theorem]{Remark}
\newtheorem{definition}[theorem]{Definition}
\newtheorem{example}[theorem]{Example}

\DeclareMathOperator{\ctg}{ctg}
\DeclareMathOperator{\col}{col}
\DeclareMathOperator{\diag}{diag}
\DeclareMathOperator{\codiag}{codiag}
\DeclareMathOperator{\dom}{dom}
\DeclareMathOperator{\Span}{span}
\DeclareMathOperator{\Vand}{Vand}

\renewcommand{\Im}{{\rm Im}}
\renewcommand{\Re}{{\rm Re}}
\newcommand{\wt}{\widetilde}

\newcommand{\eps}{\varepsilon}
\def\cH{\mathcal{H}}
\def\cP{\mathcal{P}}
\def\fH{\mathfrak{H}}
\def\fN{\mathfrak{N}}
\def\bC{\mathbb{C}}
\def\bR{\mathbb{R}}

\begin{document}
\sloppy

\title[Spectral function of ordinary differential operator]
	{Spectral functions of the simplest even order \\ ordinary
differential operator}

\author{Anton A. Lunyov}
\address{Institute of Applied Mathematics and Mechanics, NAS of Ukraine,
R. Luxemburg str. 74, \, 83114 Donetsk, Ukraine}
\curraddr{}
\email{A.A.Lunyov@gmail.com}
\thanks{The author expresses his gratitude to Prof.~M.~Malamud
for posing the problem and permanent attention to the work.}

\date{}

\begin{abstract}
We consider the minimal differential operator $A$ generated in
$L^2(0,\infty)$ by the differential expression $l(y) = (-1)^n
y^{(2n)}$. Using the technique of boundary triplets and the
corresponding Weyl functions, we find explicit form of the
characteristic matrix and the corresponding spectral function
for the Friedrichs and Krein extensions of the operator $A$.
\end{abstract}

\maketitle

\section{Introduction}
%
%
Let $\cP$ be the minimal symmetric operator, generated in
$L^2(0,\infty)$ by a differential expression
\begin{equation} \label{eq:l(y).gen}
    \sum_{k=0}^n (-1)^k \left(p_{n-k}(x) y^{(k)}\right)^{(k)},
\end{equation}
%
%
Assume that its deficiency indices are $n_{\pm}(\cP) = n$. It is
well-known~\cite[Theorem VI.21.2]{Nai69},~\cite[Theorem
II.9.1]{LevSar88} that any its proper self-adjoint extension
$\wt{\cP}$ is unitary equivalent to the multiplication operator
$\Lambda_{\sigma}$ in the space $L^2_{\sigma}(\bR)$, where
$\Lambda_{\sigma} : f(x) \to x f(x)$, $f \in L^2_\sigma(\bR)$,
and $\sigma(\cdot)$ is non-decreasing left-continuous
self-adjoint $n \times n$ matrix-function. The matrix-function
$\sigma(\cdot)$ is called a spectral function of the operator
$\wt{\cP}$ and coincides with the spectral function of the
characteristic matrix of $\wt{\cP}$, which, in turn, can be
found by the Green function of the operator $\wt{\cP}$
(see~\cite[VI.21.4]{Nai69}).

The purpose of this paper is to find the explicit form of the
spectral function for the Friedrichs extension (so-called "hard"
extension) $A_F$ and for the Krein extension $A_K$
(see~\cite[\S109]{AkhGla81} for precise definitions) of the
minimal symmetric operator $A$ generated in $L^2(0,\infty)$ by
the differential expression
\begin{equation} \label{eq:l(y)0}
    l(y) := (-1)^n y^{(2n)}(\cdot).
\end{equation}

Explicit form of the spectral function of some selfadjoint
extension $\wt{A}$ of $A$ plays important role when general
selfadjoint differential operator is treated as a perturbation
of $\wt{A}$. It is well-known that the Friedrichs extension
$A_F$ of the operator $A$ is defined by the boundary conditions
$y(0) = y'(0) = \ldots = y^{(n-1)} = 0$ and we show that the
Krein extension $A_K$ is defined by the boundary conditions
$y^{(n)}(0) = \ldots = y^{(2n-1)} = 0$.

We will exploit the technique of boundary triplets and the
corresponding Weyl functions (see Section~\ref{sec:prelim} for
precise definitions) to find the spectral function. This new
approach to extension theory of symmetric operators has been
appeared and elaborated during the last three decades
(see~\cite{Gorb91,DerMal91,DerMal92} and references therein). It
is well-known~\cite{DerMal92} that the characteristic matrix of
the selfadjoint extension $\wt{A}$ of $A$ coincides with the
Weyl function of the corresponding boundary triplet. This allows
us to find the characteristic matrix and its spectral function
easier than by classical method.

Let us formulate the main results of the paper.
%
%
\begin{theorem} \label{th:Fried}
The characteristic matrix (the Weyl function) of the Friedrichs
extension $A_F$ of the operator $A$ is given by
\begin{equation} \label{eq:MF(l)}
    M_F(\lambda) = \left(  \frac{-C_j \cdot C_k}{\sin((j+k+1)\alpha)}
    \cdot \left(\sqrt[2n]{-\lambda}\right)^{j+k+1} \right)_{j,k=0}^{n-1}, \quad \Im \lambda > 0.\\
\end{equation}
where
\begin{equation} \label{eq:C0.Cj}
    C_0 := 1, \qquad C_k := \prod_{p=1}^{k} \ctg (p\alpha),
    \quad \alpha = \frac{\pi}{2n}, \quad k \in \{1,\ldots,n-1\},
\end{equation}
and
\begin{equation} \label{eq:sqrt(-l)}
    \sqrt[2n]{-\lambda} := \sqrt[2n]{r} \cdot e^{\frac{i (\varphi - \pi)}{2n}},
    \qquad \lambda = r e^{i \varphi}, \quad 0 < \varphi < \pi.
\end{equation}
The corresponding spectral function is
\begin{eqnarray}
    \label{eq:sigma.F(t>0)}
        \sigma_F(t) &=& \frac{2n}{\pi}\left(  \frac{C_j \cdot C_k}{2n+1+j+k}
        \cdot t^{\frac{2n+1+j+k}{2n}}\right)_{j,k=0}^{n-1}, \quad t \geqslant 0, \\
    \label{eq:sigma.F(t<0)}
        \sigma_F(t) &=& 0, \quad t < 0.
\end{eqnarray}
\end{theorem}
%
%
\begin{theorem} \label{th:Krein}
The Krein extension $A_K$ of the operator $A$ is defined by the
boundary conditions
\begin{equation} \label{eq:Krein.cond}
    y^{(n)}(0) = y^{(n+1)}(0) = \ldots = y^{(2n-1)} = 0.
\end{equation}
Its characteristic matrix is
\begin{equation} \label{eq:MK(l)}
    M_K(\lambda) = \left(  \frac{-C_j \cdot C_k}{\sin((j+k+1)\alpha)}
    \cdot \left(\frac{-1}{\sqrt[2n]{-\lambda}}\right)^{j+k+1} \right)_{j,k=0}^{n-1}, \quad \Im \lambda > 0.
\end{equation}
The corresponding spectral function is
\begin{eqnarray}
    \label{eq:sigma.K(t>0)}
        \sigma_K(t) &=& \frac{2n}{\pi}\left(  (-1)^{j+k}\frac{C_j \cdot C_k}{2n-1-j-k}
        \cdot t^{\frac{2n-1-j-k}{2n}}\right)_{j,k=0}^{n-1}, \quad t \geqslant 0, \\
    \label{eq:sigma.K(t<0)}
        \sigma_K(t) &=& 0, \quad t < 0.
\end{eqnarray}
\end{theorem}
%
%
%
%
\section{Preliminaries} \label{sec:prelim}
%
%
\subsection{$R$-functions} \label{subsec:R-func}
%
%
Let $F(z)$ be $n \times n$ matrix-function defined in $\bC_+ :=
\{\lambda : \Im \lambda > 0\}$. It is called $R$-function (or
Nevanlinna function) if it is holomorphic in $\bC_+$ and $\Im
F(z) \geqslant 0$, $z \in \bC_+$. Each $R$-function admits the
following integral representation
\begin{equation} \label{eq:F(z)=A+zB+...}
    F(z) = A + z B + \int_{-\infty}^{+\infty}\left(\frac{1}{t-z}-\frac{t}{1+t^2}\right) d \sigma(t), \quad z \in \bC_+,
\end{equation}
where $A, B \in \bC^{n \times n}$ are selfadjoint matrices, $B
\geqslant 0$ and $\sigma(t)$ is non-decreasing left-continuous
selfadjoint $n \times n$ matrix-function such that the matrix
integral
\begin{equation}
    \int_{-\infty}^{+\infty} \frac{d \sigma(t)}{1+t^2}
\end{equation}
converges. The matrix-function $\sigma(\cdot)$ is called the
spectral function of $F(\cdot)$. Note that the spectral function
$\sigma(\cdot)$ of $F(\cdot)$ can be obtained by the Stieltjes
inversion formula:
\begin{equation} \label{eq:Stielt}
    \frac{1}{2}(\sigma(t+0) + \sigma(t)) - \frac{1}{2}(\sigma(s+0) + \sigma(s))
    = \frac{1}{\pi} \lim_{y \downarrow 0} \int^t_s \Im (F(x + iy)) dx, \quad s, t \in \bR.
\end{equation}
%
%
\subsection{Boundary triplets and Weyl functions} \label{subsec:BT+Weyl}
%
%
Let $A$ be a closed symmetric operator in a Hilbert space $\fH$
with equal deficiency indices $n_+(A) = n_-(A)$.
%
%
\begin{definition} \label{def:bound.triplet} \emph{(\cite{Gorb91})}
A triplet $\Pi = \{ \cH, \Gamma_0, \Gamma_1\}$ consisting of an
auxiliary Hilbert space $\cH$ and linear mappings
\begin{equation} \label{eq:Gamma.j}
    \Gamma_j: \dom(A^*) \longrightarrow \cH, \quad j \in \{0,1\},
\end{equation}
is called a boundary triplet for the adjoint operator $A^*$ of
$A$ if the following two conditions are satisfied:

(i) The second Green's formula
\begin{equation} \label{eq:Green.form}
    (A^*f,g) - (f,A^*g) = (\Gamma_1 f, \Gamma_0 g) - (\Gamma_0 f, \Gamma_1 g), \quad f,g \in \dom(A^*),
\end{equation}
takes place and

(ii) the mapping
\begin{equation} \label{eq:Gamma.def}
    \Gamma : \dom(A^*) \longrightarrow \cH \oplus \cH,
    \quad \Gamma f := \{\Gamma_0 f, \Gamma_1 f\},
\end{equation}
is surjective.
\end{definition}
%
%
It is easily seen that for each self-adjoint extension $\wt{A}$
of $A$ there exists (non-unique) boundary triplet $\Pi =
\{\cH,\Gamma_0,\Gamma_1\}$ such that
\begin{equation*}
    \dom(\wt{A}) = \ker(\Gamma_0).
\end{equation*}
We say in this case that the triplet $\Pi$ corresponds to
$\wt{A}$.
%
%
\begin{definition}\emph{(\cite{DerMal91,DerMal92}) \label{def:Weyl}}
Let $\{\cH,\Gamma_0,\Gamma_1\}$ be a boundary triplet for the
operator $A^*$ and $A_0 := A^* \upharpoonright \ker (\Gamma_0)$.
The Weyl function of $A$ corresponding to the boundary triplet
$\{\cH, \Gamma_0,\Gamma_1\}$ is the unique mapping $M(\cdot):
\rho(A_0) \longrightarrow [\cH]$ satisfying
\begin{equation} \label{eq:M(z).def}
    \Gamma_1 f_z = M(z) \Gamma_0 f_z, \quad f_z \in \fN_z := \ker(A^* - zI),
    \quad z \in \rho(A_0).
\end{equation}
\end{definition}
%
%
It is well known (see~\cite{DerMal91}) that the above implicit
definition of the Weyl function is correct and the Weyl function
$M(\cdot)$ is a $R$-function obeying $0\in \rho(\Im(M(i)))$.
Therefore, if $\dim \cH < \infty$, it admits integral
representation~\eqref{eq:F(z)=A+zB+...}, where $\sigma_M(\cdot)$
can be found by~\eqref{eq:Stielt}.
%
%
\section{Proofs of the main results} \label{sec:proofs}
%
%
\begin{lemma} \label{lem:Nl}
Let $\Im \lambda > 0$ and $\lambda = r e^{i \varphi}$, $0 <
\varphi < \pi$. Then
\begin{equation}
    \fN_{\lambda} = \Span\{y_k(\cdot,\lambda)\}_{k=0}^{n-1},
    \quad y_k(x,\lambda) := e^{\omega_k \rho x},
\end{equation}
where $\rho := i \sqrt[2n]{\lambda} := \sqrt[2n]{r} \cdot
e^{\frac{(\pi n + \varphi)i}{2n}}$ and $\omega_k := e^{\frac{i
\pi k}{n}}$.
\end{lemma}
%
%
\begin{proof}
The system $\{y_k(\cdot,\lambda)\}_{k=0}^{2n-1}$ forms a
fundamental system of solutions of equation $(-1)^n y^{(2n)} =
\lambda y$ for $\lambda \ne 0$. For $k \in \{0,1,\ldots,n-1\}$
we have
\begin{equation}
    \Re (\omega_k \rho) = \sqrt[2n]{r} \cos\left(\frac{\pi}{2} + \frac{\varphi}{2n} + \frac{\pi k}{n}\right) < 0.
\end{equation}
Hence $y_k(\cdot,\lambda) \in \fN_{\lambda}$, $k \in
\{0,1,\ldots,n-1\}$. Since $\dim \fN_{\lambda} = n$, we are
done.
\end{proof}
Let $x_0,\ldots,x_{n-1} \in \bC$. Put
\begin{equation}
    \Vand(x_0,\ldots,x_{n-1}) := (x_k^{n-1-j})_{j,k=0}^{n-1} = \begin{pmatrix}
        x_0^{n-1} & x_1^{n-1} & \cdots & x_{n-1}^{n-1} \\
        \vdots & \vdots & \cdots & \vdots\\
        x_0 & x_1 & \cdots & x_{n-1} \\
        1 & 1 & \cdots & 1
    \end{pmatrix}.
\end{equation}
The determinant of this matrix coincides with the Vandermond
determinant:
\begin{equation} \label{eq:det.Vander}
    \det(\Vand(x_0,\ldots,x_{n-1})) = \det\left((x_k^{n-1-j})_{j,k=0}^{n-1}\right) = \prod_{0 \leqslant j < k < n} (x_j - x_k).
\end{equation}
Next put
\begin{equation}
    \codiag(x_0,x_1,\ldots,x_n) = \codiag(x_j)_{j=0}^{n-1} = \begin{pmatrix}
        0 & 0 & \cdots & 0 & x_0 \\
        0 & 0 & \cdots & x_1 & 0 \\
        \vdots & \vdots & \cdots & \vdots & \vdots \\
        0 & x_{n-2} & \cdots & 0 & 0 \\
        x_{n-1} & 0 & \cdots & 0 & 0
    \end{pmatrix}.
\end{equation}
It is clear that
\begin{eqnarray}
    \label{eq:codiag.A}
    \codiag(x_j)_{j=0}^{n-1} \cdot (a_{j,k})_{j,k=0}^{n-1} &=& (x_j a_{n-1-j,k})_{j,k=0}^{n-1}, \\
    \label{eq:A.codiag}
    (a_{j,k})_{j,k=0}^{n-1} \cdot \codiag(x_k)_{k=0}^{n-1} &=& (a_{j,n-1-k} x_{n-1-k})_{j,k=0}^{n-1}.
\end{eqnarray}
\begin{proof}[Proof of Theorem~\ref{th:Fried}]
The triplet $\Pi = \{\bC^n, \Gamma_0, \Gamma_1\}$ with
\begin{eqnarray}
    \label{eq:G0} \Gamma_0 y &:=& \col(y^{(n-1)}(0), \ldots, y'(0), y(0)),\\
    \label{eq:G1} \Gamma_1 y &:=& \col(y^{(n)}(0), -y^{(n+1)}(0), \ldots, (-1)^{n-1} y^{(2n-1)}(0)),
\end{eqnarray}
is a boundary triplet for the adjoint operator $A^*$
(see~\cite{DerMal91}). Clearly it corresponds to $A_F$. Hence
the characteristic matrix of $A_F$ coincides with the Weyl
function $M_F(\lambda)$ of $A$ that corresponds to the triplet
$\Pi$.

It follows from $y^{(j)}_k(0,\lambda) = (\rho \cdot \omega_k)^j$
that
\begin{eqnarray}
    \label{eq:N0(l)}
    N_0(\lambda) &:=& \begin{pmatrix} \Gamma_0 y_0 & \ldots & \Gamma_0 y_{n-1}\end{pmatrix}
    = \left((\rho \cdot \omega_k)^{n-1-j}\right)_{j,k=0}^{n-1}. \\
    \label{eq:N1(l)}
    N_1(\lambda) &:=& \begin{pmatrix} \Gamma_1 y_0 & \ldots & \Gamma_1 y_{n-1}\end{pmatrix}
    = \left((-1)^{j}(\rho \cdot \omega_k)^{n+j}\right)_{j,k=0}^{n-1}.
\end{eqnarray}
Put
\begin{equation} \label{eq:V.def}
    V := \left(v_{jk}\right)_{j,k=0}^{n-1} := (\omega_k^{n-1-j})_{j,k=0}^{n-1}
    = \Vand(\omega_0,\ldots,\omega_{n-1}).
\end{equation}
Since numbers $\omega_0,\ldots,\omega_{n-1}$ are distinct, it
follows from~\eqref{eq:det.Vander} that $V$ is non-singular
matrix. Put $V^{-1} =: (\wt{v}_{jk})_{j,k=0}^{n-1}$. Then by
Lemma~\ref{lem:Nl}, for the Weyl function $M_F(\lambda)$ we have
\begin{eqnarray} \label{eq:M=N1.N0}
    M_F(\lambda) = N_1(\lambda) \cdot N_0^{-1}(\lambda) &=& \Bigl((-1)^{j}(\rho \cdot \omega_p)^{n+j}\Bigr)_{j,p=0}^{n-1}
    \cdot \Bigl(\rho^{k+1-n} \cdot \wt{v}_{pk}\Bigr)_{p,k=0}^{n-1} \nonumber \\
    &=& \Bigl((-1)^{j} \rho^{j+k+1} \sum_{p=0}^{n-1}
    \omega_p^{n+j} \cdot \wt{v}_{pk}\Bigr)_{j,k=0}^{n-1}.
\end{eqnarray}
Let $V_{jk}$ be the cofactor of the element $v_{jk}$ of the
matrix $V$. Combining Cramer's rule with the expansion of the
determinant according to the $k$-th row yields
\begin{equation} \label{eq:sum.omega}
    \sum_{p=0}^{n-1} \omega_p^{n+j} \cdot \wt{v}_{pk}
    = \frac{1}{\det(V)}\sum_{p=0}^{n-1} \omega_p^{n+j} V_{kp}
    = \frac{\det\left(V^{(k)}_{j}\right)}{\det(V)},
\end{equation}
where the matrix $V_{j}^{(k)}$ is obtained from the matrix $V$
by replacing the row $(\omega_p^{n-1-k})_{p=0}^{n-1}$ by the row
$(\omega_p^{n+j})_{p=0}^{n-1}$. Since $\omega_p^q = e^{\frac{\pi
i p q}{n}} = \omega_q^p$, then $V_{j}^{(k)}$ is symmetric to the
matrix $\Vand(\omega_0,\ldots,\omega_{n-k-2}, \omega_{n+j},
\omega_{n-k}, \ldots, \omega_{n-1})$ with respect to the
off-diagonal. Hence
\begin{equation} \label{eq:det.Vjk}
    \det\left(V_{j}^{(k)}\right) =
    \det\left(\Vand(\omega_0,\ldots,\omega_{n-k-2}, \omega_{n+j}, \omega_{n-k}, \ldots, \omega_{n-1})\right).
\end{equation}
Combining~\eqref{eq:det.Vander} with~\eqref{eq:det.Vjk} yields
\begin{equation*}
    \frac{\det\left(V^{(k)}_{j}\right)}{\det(V)}
    = \prod_{p=0 \atop{p \ne n-1-k}}^{n-1} \frac{\omega_{n+j} - \omega_p}{\omega_{n-1-k} - \omega_p}.
\end{equation*}
Since $\omega_q - \omega_p = 2i \eps^{p+q} \sin((q-p)\alpha)$,
where $\alpha = \frac{\pi}{2n}$ and $\eps = e^{i \alpha}$, then
\begin{eqnarray} \label{eq:det.Vjk/det.V}
    \frac{\det\left(V^{(k)}_{j}\right)}{\det(V)}
    &=& \eps^{(j+k+1)(n-1)} \cdot \prod_{p=0 \atop{p \ne n-1-k}}^{n-1}
    \frac{\sin((n+j-p)\alpha)}{\sin((n-1-k-p)\alpha)} \nonumber \\
    &=& \frac{\eps^{(j+k+1)(n-1)}}{\sin((j+k+1)\alpha)} \cdot
    \frac{\prod_{p=0}^{n-1} \cos((j-p)\alpha)}{\prod_{p=1}^{n-1-k} \sin p\alpha
    \cdot \prod_{p=1}^k (-\sin p\alpha)} \nonumber \\
    &=& \frac{(-1)^k \eps^{(j+k+1)(n-1)}}{\sin((j+k+1)\alpha)} \cdot
    \frac{\prod_{p=1}^j \cos p\alpha \cdot \prod_{p=1}^{n-1-j} \cos p\alpha}
    {\prod_{p=1}^k \sin p\alpha \cdot
    \prod_{p=1}^{n-1-k} \sin p\alpha} \nonumber \\
    &=& \frac{(-1)^k \eps^{(j+k+1)(n-1)}}{\sin((j+k+1)\alpha)} \cdot
    \prod_{p=1}^{k} \ctg p\alpha \cdot \prod_{p=1}^{j} \ctg p\alpha.
\end{eqnarray}
The last step is implied by the identity
\begin{equation} \label{eq:trig}
    \prod_{p=1}^j \cos p\alpha \cdot \prod_{p=1}^{n-1-j} \sin p\alpha
    = \prod_{p=1}^{n-1} \cos p\alpha = \prod_{p=1}^{n-1} \sin p\alpha,
    \qquad j \in \{0,1,\ldots,n-1\}.
\end{equation}
Inserting
formulas~\eqref{eq:sum.omega},~\eqref{eq:det.Vjk/det.V}
into~\eqref{eq:M=N1.N0} and taking into account the identity
$-\eps^{n-1} \rho = \sqrt[2n]{r} \cdot e^{\frac{i (\varphi -
\pi)}{2n}}$ we get the desired formula~\eqref{eq:MF(l)} for
$M_F(\lambda)$.

Now let's prove
formulas~\eqref{eq:sigma.F(t>0)}--\eqref{eq:sigma.F(t<0)}. Since
$M_F(\lambda)$ is continuous function of $\lambda$ in the closed
upper halfplane, Stieltjes inversion formula~\eqref{eq:Stielt}
and Lebesque limit theorem yields
\begin{equation} \label{eq:StieltF}
    \sigma_F(t) = \frac{1}{\pi} \int_0^t \Im \left( \lim_{y \downarrow 0} M_F(x+iy) \right) dx,
    \quad t \in \bR.
\end{equation}
Note that if $\lambda = x + iy$ with $x \in \bR$, $y > 0$,
then~\eqref{eq:sqrt(-l)} implies
\begin{equation} \label{eq:lim.sqn.-l}
    \lim_{y \downarrow 0} \sqrt[2n]{-\lambda} = \begin{cases}
        \sqrt[2n]{x} \cdot e^{-i \alpha},& \quad x \geqslant 0, \\
        \sqrt[2n]{-x},& \quad x < 0.
    \end{cases}
\end{equation}
Combining~\eqref{eq:MF(l)} with~\eqref{eq:lim.sqn.-l} yields
\begin{equation} \label{eq:lim.MF(l)}
    \lim_{y \downarrow 0} M_F(x + iy) = \begin{cases}
        \left(  -C_j \cdot C_k \cdot x^{\frac{j+k+1}{2n}} \cdot \frac{e^{-i (j+k+1) \alpha}}{\sin((j+k+1)\alpha)} \right)_{j,k=0}^{n-1}, & \quad x \geqslant 0, \\
        \left(  -C_j \cdot C_k \cdot (-x)^{\frac{j+k+1}{2n}} \cdot \frac{1}{\sin((j+k+1)\alpha)} \right)_{j,k=0}^{n-1}, & \quad x < 0.
    \end{cases}
\end{equation}
Hence
\begin{equation} \label{eq:Im.lim.MF(l)}
    \Im \left( \lim_{y \downarrow 0} M_F(x + iy) \right) = \begin{cases}
        \left(  C_j \cdot C_k \cdot x^{\frac{j+k+1}{2n}} \right)_{j,k=0}^{n-1}, & \quad x \geqslant 0, \\
        0, & \quad x < 0.
    \end{cases}
\end{equation}
Combining~\eqref{eq:StieltF} with~\eqref{eq:Im.lim.MF(l)}
yields~\eqref{eq:sigma.F(t>0)}--\eqref{eq:sigma.F(t<0)}.
\end{proof}
%
%
\begin{remark} \label{rem:const}
Calculation similar
to~\eqref{eq:M=N1.N0}--\eqref{eq:det.Vjk/det.V} was made in the
proof of Theorem 1 and Corollary 1 in~\cite{LunOri09} in
connection with sharp constants in inequalities for intermediate
derivatives. Moreover, it is curious to note that these
constants are connected with diagonal entries of the Weyl
functions $M_F(\lambda)$ and $M_K(\lambda)$. Namely, if
$A_{n,j}$, $j \in \{0,1,\ldots,n-1\}$, is the sharp constant in
the following inequality
\begin{equation}
    |f^{j}(0)| \leqslant A_{n,j} \cdot \left(\|f\|_2^2 + \|f^{(n)}\|_2^2\right)^{1/2}, \quad f \in W^{n,2}[0,\infty),
\end{equation}
then formula (1.4) from~\cite{LunOri09} and
formulas~\eqref{eq:lim.MF(l)},~\eqref{eq:lim.MK(l)} imply
\begin{equation}
    A_{n,j}^2 = \left[M_K(-1)\right]_{jj} = -\left[M_F(-1)\right]_{jj}.
\end{equation}
\end{remark}
%
%
\begin{remark} \label{rem:V-1}
Formula~\eqref{eq:MF(l)} could be also proved using explicit
formula for the inverse matrix $V^{-1}$ from~\cite{Lun07} and
some auxiliary trigonometric identity from~\cite{Lun07}. But
this way is quite cumbersome.
\end{remark}
%
%
\begin{example}
For $n=1$ the Weyl function $M_F(\lambda)$ and its spectral
function $\sigma_F(t)$ are well-known
$($see~\cite[\S132]{AkhGla81}, \cite{Nai69}$)$ and given by
\begin{equation}
    M_F(\lambda) = i \sqrt{\lambda}, \qquad \sigma_F(t) = \frac{2}{3\pi} t^{3/2}, \quad t>0,
\end{equation}
which coincides with
formulas~\eqref{eq:MF(l)},~\eqref{eq:sigma.F(t>0)} for $n=1$.
For $n=2$ these formulas turn into
\begin{equation}
    M_F(\lambda) = \begin{pmatrix}
        (i-1) \lambda^{1/4} & i \lambda^{1/2} \\
        i \lambda^{1/2} & (i+1) \lambda^{3/4}
    \end{pmatrix}, \qquad \sigma_F(t) = \frac{1}{\pi} \begin{pmatrix}
        \frac45 t^{5/4} & \frac23 t^{3/2} \\
        \frac23 t^{3/2} & \frac47 t^{7/4}
    \end{pmatrix}, \quad t>0,
\end{equation}
while for $n=3$ we have
\begin{eqnarray}
    M_F(\lambda) &=& \begin{pmatrix}
     \left(i-\sqrt{3}\right) \lambda^{1/6} & \left(-1+i \sqrt{3}\right) \lambda^{1/3} & i \lambda^{1/2} \\
     \left(-1+i \sqrt{3}\right) \lambda^{1/3} & 3 i \lambda^{1/2} & \left(1+i \sqrt{3}\right) \lambda^{2/3} \\
     i \lambda^{1/2} & \left(1+i \sqrt{3}\right) \lambda^{2/3} & \left(i+\sqrt{3}\right) \lambda^{5/6}
    \end{pmatrix}, \\
    \sigma_F(t) &=& \frac{1}{\pi}\begin{pmatrix}
        \frac67 t^{7/6} & \frac{3\sqrt{3}}{4} t^{4/3} & \frac23 t^{3/2} \\
        \frac{3\sqrt{3}}{4} t^{4/3} & 2 t^{3/2} & \frac{3\sqrt{3}}{5} t^{5/3} \\
        \frac23 t^{3/2} & \frac{3\sqrt{3}}{5} t^{5/3} & \frac{6}{11} t^{11/6}
    \end{pmatrix}, \quad t>0.
\end{eqnarray}
\end{example}
%
%
\begin{proof}[Proof of Theorem~\ref{th:Krein}]
%
By~\cite[Proposition 5]{DerMal91}, $\dom (A_K) = \ker(\Gamma_1 -
M_F(0) \Gamma_0)$, where $M_F(0) = s\text{-}\lim_{x \uparrow 0}
M_F(x)$ and $\Gamma_0$, $\Gamma_1$ are given
by~\eqref{eq:G0}--\eqref{eq:G1}. In view of~\eqref{eq:MF(l)},
$M_F(0) = 0$. Hence $\dom (A_K) = \ker(\Gamma_1)$ and the
boundary triplet $\Pi' := \{\bC^n, \Gamma_0', \Gamma_1'\} :=
\{\bC^n, \Gamma_1, -\Gamma_0\}$ corresponds to $A_K$. Definition
of $\Gamma_1$ (see~\eqref{eq:G1}) implies that $A_K$ is defined
by the boundary conditions~\eqref{eq:Krein.cond}. Also note that
\begin{equation} \label{eq:MK=-MF-1}
    M_K(\lambda) = -N_0(\lambda) N_1^{-1}(\lambda) = -M_F^{-1}(\lambda).
\end{equation}
It follows from~\eqref{eq:N0(l)} and~\eqref{eq:N1(l)} that
\begin{equation} \label{eq:N0.N1}
    N_0(\lambda) = D_0(\lambda) V, \qquad  N_1(\lambda) = D_1(\lambda) \cdot (\omega_k^j)_{j,k=0}^{n-1} \cdot D,
\end{equation}
where
\begin{gather}
    \label{eq:D0,D1.def}
    D_0(\lambda) := \diag(\rho^{n-1-j})_{j=0}^{n-1}, \qquad
    D_1(\lambda) := \diag((-1)^j \rho^{n+j})_{j=0}^{n-1}, \\
    \label{eq:D.def}
    D := \diag(\omega_k^n)_{k=0}^n = \diag((-1)^k)_{k=0}^n.
\end{gather}
Combining~\eqref{eq:codiag.A} with~\eqref{eq:V.def} yields
\begin{equation*}
    (\omega_k^j)_{j,k=0}^{n-1} = R \cdot V, \qquad R = \codiag(1,\ldots,1).
\end{equation*}
Therefore,
\begin{equation} \label{eq:N1=D1RVD}
    N_1(\lambda) = D_1(\lambda) \cdot R \cdot V \cdot D.
\end{equation}
Combining~\eqref{eq:M=N1.N0} with~\eqref{eq:N0.N1}
and~\eqref{eq:N1=D1RVD} and taking into account that $D =
D^{-1}$ and $R = R^{-1}$ we get
\begin{eqnarray}
    \label{eq:MF(l)=VDV}
    M_F(\lambda) =& N_1(\lambda) N_0^{-1}(\lambda) &= D_1(\lambda) R \cdot V D V^{-1} \cdot D_0^{-1}(\lambda), \\
    \label{eq:MK(l)=VDV}
    M_K(\lambda) =& -M_F^{-1}(\lambda) &= - D_0(\lambda) \cdot V D V^{-1} \cdot R D_1^{-1}(\lambda).
\end{eqnarray}
Expressing $VDV^{-1}$ from~\eqref{eq:MF(l)=VDV} and inserting it
to~\eqref{eq:MK(l)=VDV} we arrive at
\begin{equation} \label{eq:M0=-D0RD1...}
    M_K(\lambda) = -D_0(\lambda) R D_1^{-1}(\lambda) \cdot M_F(\lambda) \cdot D_0(\lambda) R D_1^{-1}(\lambda).
\end{equation}
Definition of $D_0(\lambda)$ and $D_1(\lambda)$
(see~\eqref{eq:D0,D1.def}) and
formulas~\eqref{eq:codiag.A}--\eqref{eq:A.codiag} implies
\begin{eqnarray} \label{eq:D0RD1}
    D_0(\lambda) R D_1^{-1}(\lambda) &=& \codiag(\rho^{n-1-j})_{j=0}^{n-1} \cdot \diag((-1)^j \rho^{-n-j})_{j=0}^{n-1} \nonumber \\
    &=& \rho^{-n} \codiag((-1)^{n-1-j})_{j=0}^{n-1}.
\end{eqnarray}
Combining~\eqref{eq:M0=-D0RD1...},~\eqref{eq:D0RD1},~\eqref{eq:codiag.A},~\eqref{eq:A.codiag}
and~\eqref{eq:MF(l)} yields
\begin{eqnarray} \label{eq:MK(l).final}
    M_K(\lambda) &=& \left(\rho^{-2n}(-1)^{n-1-j+k} \frac{C_{n-1-j} \cdot C_{n-1-k}}{\sin((2n-1-j-k)\alpha)}
    \left(\sqrt[2n]{-\lambda}\right)^{2n-1-j-k}\right)_{j,k=0}^{n-1} \nonumber \\
    &=& \left(-\lambda \cdot \frac{i^{-2n}}{\lambda} \cdot (-1)^n \cdot \frac{C_{n-1-j} \cdot C_{n-1-k}}{\sin((j+k+1)\alpha)}
    \left(\frac{-1}{\sqrt[2n]{-\lambda}}\right)^{j+k+1}\right)_{j,k=0}^{n-1}
\end{eqnarray}
It follows from~\eqref{eq:trig} that $C_j = C_{n-1-j}$, $j \in
\{0,1,\ldots,n-1\}$. In view of this,~\eqref{eq:MK(l).final}
implies the desired formula~\eqref{eq:MK(l)} for $M_K(\lambda)$.

Now let's prove
formulas~\eqref{eq:sigma.K(t>0)}--\eqref{eq:sigma.K(t<0)}. It
follows from~\eqref{eq:MK(l)} that for
$j,k\in\{0,1,\ldots,n-1\}$
\begin{equation}
    \bigl|[M_K(x+iy)]_{jk}\bigr| \leqslant C \left(|x|^{-1+\frac{1}{2n}} + |x|^{-\frac{1}{2n}}\right),
    \qquad x \in \bR \setminus \{0\}, \quad y > 0,
\end{equation}
for some $C>0$. Hence Stieltjes inversion
formula~\eqref{eq:Stielt} and Lebesque limit theorem yields
\begin{equation} \label{eq:StieltK}
    \sigma_K(t) = \frac{1}{\pi} \int_0^t \Im \left( \lim_{y \downarrow 0} M_K(x+iy) \right) dx,
    \quad t \in \bR.
\end{equation}
Combining~\eqref{eq:MK(l)} with~\eqref{eq:lim.sqn.-l} we arrive
at
\begin{equation} \label{eq:lim.MK(l)}
    \lim_{y \downarrow 0} M_K(x + iy) = \begin{cases}
        \left(  (-1)^{j+k} \cdot C_j \cdot C_k \cdot x^{-\frac{j+k+1}{2n}}
        \cdot \frac{e^{i (j+k+1) \alpha}}{\sin((j+k+1)\alpha)} \right)_{j,k=0}^{n-1}, & \quad x > 0, \\
        \left(  (-1)^{j+k} \cdot C_j \cdot C_k \cdot (-x)^{-\frac{j+k+1}{2n}}
        \cdot \frac{1}{\sin((j+k+1)\alpha)} \right)_{j,k=0}^{n-1}, & \quad x < 0.
    \end{cases}
\end{equation}
Hence
\begin{equation} \label{eq:Im.lim.MK(l)}
    \Im \left( \lim_{y \downarrow 0} M_K(x + iy) \right) = \begin{cases}
        \left(  (-1)^{j+k} C_j \cdot C_k \cdot x^{-\frac{j+k+1}{2n}} \right)_{j,k=0}^{n-1}, & \quad x > 0, \\
        0, & \quad x < 0.
    \end{cases}
\end{equation}
Combining~\eqref{eq:StieltK} with~\eqref{eq:Im.lim.MK(l)}
yields~\eqref{eq:sigma.K(t>0)}--\eqref{eq:sigma.K(t<0)}.
\end{proof}
%
%
\begin{remark}
Formulas~\eqref{eq:MF(l)},~\eqref{eq:MK(l)}
and~\eqref{eq:MK=-MF-1} lead to the following curious identity
\begin{equation}
    \sum_{p=0}^{n-1} \frac{(-1)^{p+k} \cdot C_j \cdot C_p^2 \cdot C_k}
    {\sin((j+p+1)\alpha) \sin((p+k+1)\alpha)} = \delta_{jk},
    \quad j,k \in \{0,1,\ldots,n-1\}.
\end{equation}
It seems non-trivial to prove it directly.
\end{remark}
%
%

%
%
\end{document}